\documentclass{amsart}

\title[The Vop\v{e}nka principle and the Vop\v{e}nka scheme]{The Vop\v{e}nka principle is inequivalent to but conservative over the Vop\v{e}nka scheme}

\author{Joel David Hamkins}
 \address[J.~D.~Hamkins]
         {Mathematics, Philosophy, Computer Science, The Graduate Center of The City University of New York, 365 Fifth Avenue, New York, NY 10016 \& Mathematics, College of Staten Island of CUNY}
\email{jhamkins@gc.cuny.edu}
\urladdr{http://jdh.hamkins.org}

\thanks{I would like to thank Victoria Gitman for helpful conversations about the \Vopenka\ principle. My research has been supported by grant \#69573-00 47 from the CUNY Research Foundation. Commentary can be made on the author's blog at \href{http://jdh.hamkins.org/vopenka-principle-vopenka-scheme}{http://jdh.hamkins.org/vopenka-principle-vopenka-scheme}.}

\usepackage{latexsym,amsfonts,amsmath,amssymb,mathrsfs}
\usepackage[hidelinks]{hyperref}
\usepackage{verbatim}

\newtheorem{theorem}{Theorem}

\newtheorem*{maintheorem*}{Main Theorem}
\newtheorem*{maintheorems*}{Main Theorems}
\newtheorem{corollary}[theorem]{Corollary}
\newtheorem*{corollary*}{Corollary}
\newtheorem*{corollaries*}{Corollaries}

\newtheorem{lemma}[theorem]{Lemma}

\newtheorem*{questions*}{Questions}
\newtheorem*{mainquestion*}{Main Question} 
\newtheorem*{openquestion*}{Open Question} 

\newcommand{\QED}{\end{proof}}

\def\proclaim[#1]{{\bf #1}}
\def\BF#1.{{\bf #1.}}

\newcommand\Vopenka{Vop\v{e}nka}

\newcommand{\Godel}{G\"odel}

\newcommand{\Levy}{L\'{e}vy}

\newcommand{\of}{\subseteq}

\newcommand{\set}[1]{\{\,{#1}\,\}}

\newcommand{\elesub}{\prec}

\newcommand{\Add}{\mathop{\rm Add}}

\newcommand{\image}{\mathbin{\hbox{\tt\char'42}}}

\newcommand{\restrict}{\upharpoonright} 
\newcommand{\satisfies}{\models}
\newcommand{\forces}{\Vdash}
\newcommand{\proves}{\vdash}

\newcommand{\union}{\cup}

\newcommand{\intersect}{\cap}

\newcommand{\trianglelt}{\lhd}

\newcommand{\smalllt}{\mathrel{\mathchoice{\raise2pt\hbox{$\scriptstyle<$}}{\raise1pt\hbox{$\scriptstyle<$}}{\raise0pt\hbox{$\scriptscriptstyle<$}}{\scriptscriptstyle<}}}
\newcommand{\smallleq}{\mathrel{\mathchoice{\raise2pt\hbox{$\scriptstyle\leq$}}{\raise1pt\hbox{$\scriptstyle\leq$}}{\raise1pt\hbox{$\scriptscriptstyle\leq$}}{\scriptscriptstyle\leq}}}
\newcommand{\lt}{\smalllt}

\newcommand{\leqlambda}{{{\smallleq}\lambda}}

\newcommand{\ltdelta}{{{\smalllt}\delta}}

\newcommand{\boolval}[1]{\mathopen{\lbrack\!\lbrack}\,#1\,\mathclose{\rbrack\!\rbrack}}
\def\[#1]{\boolval{#1}}
\newbox\gnBoxA
\newdimen\gnCornerHgt
\setbox\gnBoxA=\hbox{\tiny$\ulcorner$}
\global\gnCornerHgt=\ht\gnBoxA
\newdimen\gnArgHgt
\def\gcode #1{%
\setbox\gnBoxA=\hbox{$#1$}%
\gnArgHgt=\ht\gnBoxA%
\ifnum     \gnArgHgt<\gnCornerHgt \gnArgHgt=0pt%
\else \advance \gnArgHgt by -\gnCornerHgt%
\fi \raise\gnArgHgt\hbox{\tiny$\ulcorner$} \box\gnBoxA %
\raise\gnArgHgt\hbox{\tiny$\urcorner$}}
\newcommand{\UnderTilde}[1]{{\setbox1=\hbox{$#1$}\baselineskip=0pt\vtop{\hbox{$#1$}\hbox to\wd1{\hfil$\sim$\hfil}}}{}}
\newcommand{\Undertilde}[1]{{\setbox1=\hbox{$#1$}\baselineskip=0pt\vtop{\hbox{$#1$}\hbox to\wd1{\hfil$\scriptstyle\sim$\hfil}}}{}}
\newcommand{\undertilde}[1]{{\setbox1=\hbox{$#1$}\baselineskip=0pt\vtop{\hbox{$#1$}\hbox to\wd1{\hfil$\scriptscriptstyle\sim$\hfil}}}{}}
\newcommand{\UnderdTilde}[1]{{\setbox1=\hbox{$#1$}\baselineskip=0pt\vtop{\hbox{$#1$}\hbox to\wd1{\hfil$\approx$\hfil}}}{}}
\newcommand{\Underdtilde}[1]{{\setbox1=\hbox{$#1$}\baselineskip=0pt\vtop{\hbox{$#1$}\hbox to\wd1{\hfil\scriptsize$\approx$\hfil}}}{}}

\renewcommand{\iff}{\mathrel{\leftrightarrow}}

\def\<#1>{\left\langle#1\right\rangle}

\newcommand{\Ord}{\mathord{{\rm Ord}}}

\newcommand{\ZFC}{{\rm ZFC}}

\newcommand{\KM}{{\rm KM}}
\newcommand{\GB}{{\rm GB}}
\newcommand{\GBC}{{\rm GBC}}

\newcommand{\AC}{{\rm AC}}

\newcommand\VP{{\rm VP}}
\newcommand\VS{{\rm VS}}

\newcommand{\cell}[1]{\boxit{\hbox to 17pt{\strut\hfil$#1$\hfil}}}
\newcommand{\head}[2]{\lower2pt\vbox{\hbox{\strut\footnotesize\it\hskip3pt#2}\boxit{\cell#1}}}
\newcommand{\boxit}[1]{\setbox4=\hbox{\kern2pt#1\kern2pt}\hbox{\vrule\vbox{\hrule\kern2pt\box4\kern2pt\hrule}\vrule}}
\newcommand{\Col}[3]{\hbox{\vbox{\baselineskip=0pt\parskip=0pt\cell#1\cell#2\cell#3}}}
\newcommand{\tapenames}{\raise 5pt\vbox to .7in{\hbox to .8in{\it\hfill input: \strut}\vfill\hbox to
.8in{\it\hfill scratch: \strut}\vfill\hbox to .8in{\it\hfill output: \strut}}}
\newcommand{\Head}[4]{\lower2pt\vbox{\hbox to25pt{\strut\footnotesize\it\hfill#4\hfill}\boxit{\Col#1#2#3}}}
\newcommand{\Dots}{\raise 5pt\vbox to .7in{\hbox{\ $\cdots$\strut}\vfill\hbox{\ $\cdots$\strut}\vfill\hbox{\
$\cdots$\strut}}}
\newcommand{\df}{\it} 
\begin{document}

\begin{abstract}
 The \Vopenka\ principle, which asserts that every proper class of first-order structures in a common language admits an elementary embedding between two of its members, is not equivalent over \GBC\ to the first-order \Vopenka\ scheme, which makes the \Vopenka\ assertion only for the first-order definable classes of structures. Nevertheless, the two \Vopenka\ axioms are equiconsistent  and they have exactly the same first-order consequences in the language of set theory. Specifically, \GBC\ plus the \Vopenka\ principle is conservative over \ZFC\ plus the \Vopenka\ scheme for first-order assertions in the language of set theory.
\end{abstract}

\maketitle

\noindent
The {\df \Vopenka\ principle} is the assertion that for every proper class $\mathcal{M}$ of first-order $\mathcal{L}$-structures, for a set-sized language $\mathcal{L}$, there are distinct members of the class $M,N\in\mathcal{M}$ with an elementary embedding $j:M\to N$ between them. In quantifying over classes, this principle is a single assertion in the language of second-order set theory, and it makes sense to consider the \Vopenka\ principle in the context of a second-order set theory, such as \Godel-Bernays set theory \GBC, whose language allows one to quantify over classes. In this article, \GBC\ includes the global axiom of choice.

In contrast, the first-order {\df \Vopenka\ scheme} makes the \Vopenka\ assertion only for the first-order definable classes $\mathcal{M}$ (allowing set parameters, and henceforth in this article, by the term ``definable class,'' I shall intend that parameters are allowed). This theory can be expressed as a scheme of first-order statements, one for each possible definition of a class, and it makes sense to consider the \Vopenka\ scheme in Zermelo-Frankael \ZFC\ set theory with the axiom of choice.

Because the \Vopenka\ principle is a second-order assertion, it does not make sense to refer to it in the context of \ZFC\ set theory, whose first-order language does not allow quantification over classes; one typically retreats to the \Vopenka\ scheme in that context. The theme of this article is to investigate the precise meta-mathematical interactions between these two treatments of \Vopenka's idea.

\begin{maintheorems*}\ 
 \begin{enumerate}
  \item If \ZFC\ and the \Vopenka\ scheme holds, then there is a class forcing extension, adding classes but no sets, in which \GBC\ and the \Vopenka\ scheme holds, but the \Vopenka\ principle fails.
  \item If \ZFC\ and the \Vopenka\ scheme holds, then there is a class forcing extension, adding classes but no sets, in which \GBC\ and the \Vopenka\ principle holds.
 \end{enumerate}
\end{maintheorems*}

\noindent 
These results appear as theorems~\ref{Theorem.VS+notVP} and~\ref{Theorem.Forcing-GBC+VP}, respectively. It follows that the \Vopenka\ principle \VP\ and the \Vopenka\ scheme \VS\ are not equivalent, but they are equiconsistent and indeed, they have the same first-order consequences.

\goodbreak
\begin{corollaries*}\
 \begin{enumerate}
   \item Over \GBC, the \Vopenka\ principle and the \Vopenka\ scheme, if consistent, are not equivalent.
   \item Nevertheless, the two \Vopenka\ axioms are equiconsistent over \GBC.
   \item Indeed, the two \Vopenka\ axioms have exactly the same first-order consequences in the language of set theory. Specifically, \GBC\ plus the \Vopenka\ principle is conservative over \ZFC\ plus the \Vopenka\ scheme for assertions in the first-order language of set theory.
         $$\GBC+\VP\proves\phi\qquad\text{if and only if}\qquad\ZFC+\VS\proves\phi$$
 \end{enumerate}
\end{corollaries*}

\noindent These consequences are explained in corollaries~\ref{Corollary.VP-conservative-over-VS} and~\ref{Corollary.VP-equiconsistent-VS}.

\section{$A$-Extendible cardinals}

It turns out that the \Vopenka\ principle and the \Vopenka\ scheme admit a convenient large-cardinal characterization in terms of the class-relativized extendible cardinals, and so let us develop a little of that large-cardinal theory here, before giving the characterization in section~\ref{Section.VP-characterization}.

Namely, we define that a cardinal $\kappa$ is {\df extendible}, if for every ordinal $\lambda>\kappa$ there is an ordinal $\theta$ and an elementary embedding $j:V_\lambda\to V_\theta$ with critical point $\kappa$ and $\lambda<j(\kappa)$. Every such cardinal is, of course, a measurable cardinal, and indeed, a supercompact cardinal, since we may define the induced normal fine measures by $X\in \mu\iff j\image\beta\in j(X)$ for $X\of P_\kappa\beta$, provided $\beta+1<\lambda$.

More generally, we say that a cardinal $\kappa$ is {\df $A$-extendible} for a class $A$, if for every ordinal $\lambda>\kappa$ there is an ordinal $\theta$ and an elementary embedding $$j:\<V_\lambda,\in,A\intersect V_\lambda>\to\<V_\theta,\in,A\intersect V_\theta>$$ with critical point $\kappa$ and $\lambda<j(\kappa)$. Let us refer to $\lambda$ as the {\df degree} of $A$-extendibility of this embedding. Let's show next that the $\lambda<j(\kappa)$ requirement is an inessential convenience.

\begin{lemma}\label{Lemma.A-extendible-iff}
 A cardinal $\kappa$ is $A$-extendible if and only if for every ordinal $\lambda>\kappa$ there is an ordinal $\theta$ and an elementary embedding $j:\<V_\lambda,\in,A\intersect V_\lambda>\to\<V_\theta,\in,A\intersect V_\theta>$ with critical point $\kappa$.
\end{lemma}

\begin{proof}
The point here is that in the latter property, we have dropped the requirement that $\lambda<j(\kappa)$. For the case of extendible cardinals, where there is no predicate $A$, this equivalence is proved in Kanamori's book~\cite[proposition~23.15]{Kanamori2004:TheHigherInfinite2ed}, and one may simply carry his argument through here with the predicate $A$. But let me give an alternative slightly simpler argument. The forward implication is immediate. Conversely, suppose that $\kappa$ is $A$-extendible with respect to the weaker property, and consider any ordinal $\lambda>\kappa$. What I claim is that for every ordinal $\alpha$, there is an elementary embedding $j:\<V_\lambda,\in,A\intersect V_\lambda>\to \<V_\theta,\in,A\intersect V_\theta>$ with critical point $\kappa$ and $j(\kappa)>\alpha$. The case $\alpha=\lambda$ then proves the lemma. I shall prove the claim by induction on $\alpha$. The statement is clearly true for all ordinals $\alpha$ up to the next inaccessible cardinal above $\kappa$, since $j(\kappa)$ will be at least that large regardless. Suppose by induction that the statement is true for all $\beta<\alpha$, where $\kappa<\alpha$. Let $\bar\lambda$ be larger than $\lambda$ and $\alpha+2$ and large enough so that the embeddings witnessing the induction assumption all exist inside $V_{\bar\lambda}$. By $A$-extendibility, there is an ordinal $\theta$ and an elementary embedding $j:\<V_{\bar\lambda},\in,A\intersect V_{\bar\lambda}>\to \<V_\theta,\in,A\intersect V_\theta>$ with critical point $\kappa$. It cannot be that $\alpha$ is closed under $j$, meaning that $j\image\alpha\of\alpha$, since if this were true then the critical sequence would be entirely below $\alpha$, and the supremum of the critical sequence would be a fixed point of $j$ below or at $\alpha$, contrary to the Kunen inconsistency. So there is some ordinal $\beta$ with $\beta<\alpha<j(\beta)$. By the choice of $\bar\lambda$, there is inside $V_{\bar\lambda}$ an elementary embedding $h:\<V_\lambda,\in,A\intersect V_\lambda>\to\<V_\eta,\in,A\intersect V_\eta>$ with critical point $\kappa$ and $h(\kappa)>\beta$. Since this embedding is an element of $V_{\bar\lambda}$, we have $\eta<\bar\lambda$, and so the composition $j\circ h$ makes sense. Since $h$ is elementary from $\<V_\lambda,\in,A\intersect V_\lambda>$ to $\<V_\eta,\in,A\intersect V_\eta>$ and $j\restrict V_\eta$ is elementary from $\<V_\eta,\in,A\intersect V_\eta>$ to $\<V_{j(\eta)},\in,A\intersect V_{j(\eta)}>$, it follows that $j\circ h$ is an elementary embedding from $\<V_\lambda,\in,A\intersect V_\lambda>$ to $\<V_{j(\eta)},\in,A\intersect V_{j(\eta)}>$, with critical point $\kappa$. Since $\beta<h(\kappa)$, it follows that $j(\beta)<j(h(\kappa))=(j\circ h)(\kappa)$. Since $\alpha<j(\beta)$, we have therefore verified the desired property for $\alpha$. By induction, therefore, every ordinal $\alpha$ has this property, and so in particular, there is an elementary embedding witnessing the $A$-extendibility of $\kappa$ for which $j(\kappa)>\lambda$, as desired.
\end{proof}

Note that if $j:\<V_{\lambda+1},\in,A\intersect V_{\lambda+1}>\to\<V_{\theta+1},\in,A\intersect V_{\theta+1}>$ is an $A$-extendibility embedding of degree $\lambda+1$, then by extracting the induced extender embedding and applying that embedding to all of $V$, we may produce an extender ultrapower embedding $j:V\to M$ which agrees with the original embedding $j$ on $V_\lambda$. In particular, we'll have critical point $\kappa$, $\lambda<j(\kappa)$, $V_{j(\lambda)}\of M$ and $j(A\intersect V_\lambda)=A\intersect V_\theta$. This provides another characterization of $A$-extendibility.

If $\kappa$ is $A$-extendible, then let us define that a set $g\of\kappa$ is {\df $A$-stretchable}, if for every $\lambda>\kappa$ and every $h\of\lambda$ for which $h\intersect\kappa=g$, there is an elementary embedding $j:\<V_\lambda,\in,A\intersect V_\lambda>\to \<V_\theta,\in,A\intersect V_\theta>$ with critical point $\kappa$ and $\lambda<j(\kappa)$, for which $j(g)\intersect\lambda =h$. Thus, an $A$-stretchable set $g$ is one that can be stretched by an $A$-extendibility embedding so as to agree with any desired $h$ extending $g$.

\begin{theorem}\label{Theorem.Stretchable-set}
 If $\kappa$ is $A$-extendible for some class $A$, then there is an $A$-stretchable set $g\of\kappa$.
\end{theorem}

\begin{proof}
We define the initial segments of $g$ in stages. Fix a well-ordering $\trianglelt$ of $V_\kappa$. If $g\intersect\gamma$ is defined, then let $\lambda>\gamma$ be least such that there is some $h\of\lambda$ extending $g\intersect\gamma$, such that there is no elementary embedding $j:\<V_\lambda,\in,A\intersect V_\lambda>\to\<V_\theta,\in,A\intersect V_\theta>$ with critical point $\gamma$ and $\lambda<j(\gamma)$, for which $j(g\intersect\gamma)\intersect\lambda=h$. That is, $g\intersect\gamma$ does not anticipate $h$ with respect to any $A$-extendibility embedding of degree $\lambda$. An easy reflection argument shows that $\lambda<\kappa$. For this minimal $\lambda$, let $h$ be the $\trianglelt$-least such set that is not anticipated, and define $g\restrict\lambda=h$. This procedure defines $g$ all the way up to $\kappa$.

To see that $g$ is $A$-stretchable, suppose that there is some $\lambda$ and set $h\of\lambda$ extending $g$ for which there is no $A$-extendibility embedding $j:\<V_\lambda,\in,A\intersect V_\lambda>\to\<V_\theta,\in,A\intersect V_\theta>$ with $j(g)\intersect\lambda=h$. Let $\bar\lambda>\lambda$ be large enough so that this assertion about $\lambda,\kappa,g,h$ is absolute to $\<V_{\bar\lambda},\in,A\intersect V_{\bar\lambda}>$. By $A$-extendibility, there is an ordinal $\bar\theta$ and an elementary embedding $j:\<V_{\bar\lambda},\in,A\intersect V_{\bar\lambda}>\to \<V_{\bar\theta},\in,A\intersect V_{\bar\theta}>$ with critical point $\kappa$ and $\bar\lambda<j(\kappa)$. The set $j(g)\of j(\kappa)$ is defined by the same procedure as $g$, except using $j(\trianglelt)$, which is a well-order of $V_{j(\kappa)}$. At stages below $\kappa$, of course this procedure agrees with what happened in the construction of $g$. At stage $\kappa$, since $\bar\lambda\leq\theta$, the model will agree that $\lambda$ is least for which $g$ has no $A$-extension to some set $h'\of\lambda$ extending $g$, and so it will select the $j(\trianglelt)$-least such $h'$ and define $j(g)\intersect\lambda=h'$. So $\<V_{\bar\theta},\in,A\intersect V_{\bar\theta}>$ will believe that there is no $A$-extendibility embedding for $\kappa$ of degree $\lambda$, for which $g$ anticipates $h'$. But this contradicts the fact that $j\restrict V_\lambda$ is precisely such an embedding, and it is an element of $V_{\bar\theta}$. Therefore, $g$ must be $A$-stretchable, as desired.
\end{proof}

If one were to make a bounded number of changes to $g$ below $\kappa$, then it would not affect the stretchability feature, and consequently every bounded subset of $\kappa$ can be extended to an $A$-stretchable set $g$. I shall make use of this observation in the proof of theorem~\ref{Theorem.Forcing-GBC+VP}.

One should look upon stretchability as a form of the Laver diamond property for $A$-extendibility. (See \cite{Hamkins:LaverDiamond} for diverse versions of the Laver diamond for various large cardinals.) Specifically, if $\kappa$ is $A$-extendible, then a function $\ell:\kappa\to V_\kappa$ is an $A$-extendibility {\df Laver function}, if for every $\lambda$ and every target $a$, there is an elementary embedding $j:\<V_\lambda,\in,A\intersect V_\lambda>\to\<V_\theta,\in,A\intersect V_\theta>$ with critical point $\kappa$ and $\lambda<j(\kappa)$, for which $j(\ell)(\kappa)=a$. Essentially the same argument as in theorem~\ref{Theorem.Stretchable-set} establishes:

\begin{theorem} 
 If $\kappa$ is $A$-extendible, then there is an $A$-extendibility Laver function $\ell:\kappa\to V_\kappa$. 
\end{theorem}

\begin{proof}
This can be seen as an immediate consequence of theorem~\ref{Theorem.Stretchable-set}, simply by coding the object $a$ as a set of ordinals, and decoding $\ell$ from the stretchable set $g$. 

Alternatively, we may also undertake a direct argument. Namely, if $\ell\restrict\gamma$ is defined, then let $\lambda$ be least such that some object $a$ is not anticipated by any $A$-extendibility embedding $j:\<V_\lambda,\in,A\intersect V_\lambda>\to\<V_\theta,\in,A\intersect V_\theta>$ with critical point $\gamma$. Let $\ell(\gamma)$ be the least such $a$ with respect to a fixed well-ordering $\trianglelt$ of $V_\kappa$. A reflection argument shows that $\lambda<\kappa$ and $a\in V_\kappa$. If the resulting $\ell:\kappa\to V_\kappa$ does not have the Laver function property, let $\lambda$ be least such that there is some $a$ that is not anticipated by any $A$-extendibility embedding of degree $\lambda$. Let $\bar\lambda>\lambda$ be large enough so that this property is seen in $V_{\bar\lambda}$, and fix an elementary embedding $j:\<V_{\bar\lambda},\in,A\intersect V_{\bar\lambda}>\to\<V_{\bar\theta},\in,A\intersect V_{\bar\theta}>$ with critical point $\kappa$ and $\bar\lambda<j(\kappa)$. In $V_{\bar\theta}$, the definition of $j(\ell)(\kappa)$ will find the same $\lambda$ and pick some $a'$ not anticipated by any $A$-extendibility embedding of degree $\lambda$. But this is contradicted by the existence of $j\restrict V_\lambda$ itself, which exists in $V_{\bar\theta}$.
\end{proof}

\begin{corollary}
 In \GBC, for any class $A$ there is a class function $\ell:\Ord\to V$, such that whenever $\kappa$ is $A$-extendible, then $\ell\restrict\kappa$ is an $A$-extendible Laver function for $\kappa$.
\end{corollary}

\begin{proof}
The point is that a global well-ordering $\trianglelt$ provides a uniform method to define the Laver function $\ell$ all the way up to $\Ord$. 
\end{proof}

If one does not have global choice, then one still can construct at least an {\df ordinal-anticipating} Laver function $\ell:\Ord\to\Ord$, such that for any $A$-extendible cardinal, the function $\ell\restrict\kappa$ has the Laver function property as far as anticipating ordinals is concerned. In particular, this function will have the $A$-extendibility {\df Menas} property. 

Although I won't require this in the proof of the main theorems, let me anyway investigate the degree of reflectivity and correctness for $A$-extendible cardinals. An ordinal $\gamma$ is {\df $\Sigma_n(A)$-correct}, if $\<V_\gamma,\in,A\intersect V_\gamma>\elesub_{\Sigma_n}\<V,\in,A>$. A cardinal $\kappa$ is {\df $\Sigma_n(A)$-reflecting}, if it is inaccessible and $\Sigma_n(A)$-correct.

\begin{theorem}
 If $\kappa$ is $A$-extendible for a class $A$, then $\kappa$ is $\Sigma_2(A)$-reflecting. If $\kappa$ is $A\oplus C$-extendible, where $C$ is the class of all $\Sigma_1(A)$-correct ordinals, then $\kappa$ is $\Sigma_3(A)$-reflecting.
\end{theorem}

\begin{proof}
First, let's notice that every $A$-extendible $\kappa$ is $\Sigma_1(A)$-reflecting. Upward absoluteness is immediate. Conversely, suppose that there is some $x$ for which $\varphi(x,a)$, where all quantifiers are bounded (but there is a predicate for $A$). The witness $x$ exists in some $V_\lambda$, and so $\<V_\lambda,\in,A\intersect V_\lambda>\satisfies\varphi(x,a)$. Let $j:\<V_\lambda,\in,A\intersect V_\lambda>\to\<V_\theta,\in,A\intersect V_\theta>$ witness the $A$-extendibility of $\kappa$, with $\lambda<j(\kappa)$. Thus, $x\in V_{j(\kappa)}$, and so $\<V_{j(\kappa)},\in,A\intersect V_{j(\kappa)}>$ satisfies $\exists x\, \varphi(x,a)$. By elementarity, therefore, this is also true in $\<V_\kappa,\in,A\intersect V_\kappa>$, as desired.

Also every $A$-extendible cardinal is $\Sigma_2(A)$-reflecting. Again, the upward absoluteness is immediate by the previous paragraph. If there is some $x$ for which $\forall y\, \varphi(x,y,a)$, where $\varphi$ has only bounded quantifiers (and $A$ may appear as a predicate, then pick $\lambda$ large enough so that $x\in V_\lambda$, and again consider the embedding $j$ as above. Note again that $x\in V_\lambda\of V_{j(\kappa)}$. And since $\varphi(x,y,a)$ hold for all $y$, we see that $\<V_{j(\kappa)},\in,A\intersect V_{j(\kappa)}>\satisfies\exists x\forall y\, \varphi(x,y,a)$, and so again this pulls back to $\<V_\kappa,\in,A\intersect V_\kappa>$ by elementarity, as desired.

Finally, consider $\Sigma_3(A)$, and suppose now that $\kappa$ is $A\oplus C$-extendible, where $C$ is the class of $\Sigma_1(A)$-correct ordinals. The upward absoluteness of $\Sigma_3(A)$ from $\<V_\kappa,\in,A\intersect V_\kappa>$  follows from the $\Sigma_2(A)$-correctness of $\kappa$ established in the previous paragraph. So suppose that $\exists x\forall y\exists z\, \varphi(x,y,z,a)$ holds in $\<V,\in,A>$ for some $a\in V_\kappa$. Let $\lambda$ be large enough so that the witness $x$ is in $V_\lambda$. Let $j:\<V_\lambda,\in,A\intersect V_\lambda,C\intersect\lambda>\to\<V_\theta,\in,A\intersect V_\theta,C\intersect \theta>$ witness the $A\oplus C$-extendibility of $\kappa$. Since $\kappa\in C$, it follows that $j(\kappa)\in C$ also, and so $j(\kappa)$ is $\Sigma_1(A)$-correct. Thus, since $x\in V_\lambda\of V_{j(\kappa)}$ and $\<V,\in,A>\satisfies\forall y\exists z\, \varphi(x,y,z,a)$, it follows that $\<V_{j(\kappa)},\in,A\intersect V_{j(\kappa)}>\satisfies \forall y\exists z\, \varphi(x,y,z,a)$. Consequently, this model  satisfies the assertion $\exists x\forall y\exists z\, \varphi(x,y,z,a)$, and so we may pull back the statement to $\<V_\kappa,\in,A\intersect V_\kappa>$, thereby verifying this instance of the $\Sigma_3(A)$-correctness of $\kappa$, as desired.
\end{proof}

Note that without the predicate $A$, every $\beth$-fixed point is $\Sigma_1$-correct, since this is the \Levy\ refection theorem. Thus, when there is no $A$, we get for free that $j(\kappa)$ is $\Sigma_1$-correct, and so every extendible cardinal is $\Sigma_3$-reflecting. It is not clear to me at the moment in the general case whether every $A$-extendible cardinal is $\Sigma_3(A)$-reflecting, since the argument appears to rely on the $\Sigma_1(A)$-correctness of $j(\kappa)$.

\section{Large-cardinal characterization of \Vopenka's principle}\label{Section.VP-characterization}

I shall now give the large-cardinal characterization of the \Vopenka\ principle. 

\begin{theorem}\label{Theorem.VPequivalents} 
The following are equivalent over $\GB+\AC$ set theory. 
 \begin{enumerate}
  \item The \Vopenka\ principle.
  \item For every class $A$ and all sufficiently large $\lambda$, there is an ordinal $\theta>\lambda$ and an elementary embedding $j:\<V_\lambda,\in,A\intersect V_\lambda>\to\<V_\theta,\in,A\intersect V_\theta>$.
  \item For every class $A$, there is an $A$-extendible cardinal.
  \item For every class $A$, there is a stationary proper class of $A$-extendible cardinals.
  \item For every class $A$ and all sufficiently large ordinals $\lambda$, there is a transitive class $M$ and an elementary embedding $j:V\to M$ with some critical point $\kappa<\lambda$, such that $\lambda<j(\kappa)$ and $j(A\intersect V_\lambda)=A\intersect V_{j(\lambda)}$.
 \end{enumerate}
\end{theorem}

\begin{proof}
($1\to 2$) Assume towards contradiction that the \Vopenka\ principle holds, but statement $2$ fails for some class $A$. So there are unboundedly many ordinals $\lambda$ for which there is no elementary embedding $j:\<V_\lambda,\in,A\intersect V_\lambda>\to\<V_\theta,\in,A\intersect V_\theta>$ for any ordinal $\theta>\lambda$. Let $\mathcal{M}$ be the class of all such structures $\<V_\lambda,\in,A\intersect V_\lambda>$ having no such embedding. By the \Vopenka\ principle, there is an elementary embedding between two of these structures $j:\<V_\lambda,\in,A\intersect V_\lambda>\to\<V_\theta,\in,A\intersect V_\theta>$, with $\lambda<\theta$, contrary to the inclusion of the former structure in $\mathcal{M}$. 

($2\to 3$) Suppose that $A$ is a class and for some ordinal $\lambda_0$ and all $\lambda\geq\lambda_0$, there is an ordinal $\theta>\lambda$ and an elementary embedding $j:\<V_\lambda,\in,A\intersect V_\lambda>\to\<V_\theta,\in,A\intersect V_\theta>$. For singular $\lambda$, we may assume without loss that $j$ has a critical point below $\lambda$, by considering $j\restrict V_\lambda$ for an embedding $j$ on $V_{\lambda+1}$, which must move $\lambda$, but cannot have $\lambda$ as its critical point. So we have a critical point $\kappa<\lambda$, although different $\lambda$ could have different such critical points. Nevertheless, the map $\lambda\mapsto\kappa$, choosing the smallest such $\kappa$, is a definable pressing-down function, and so by the class version of Fodor's lemma, there is a stationary class of $\lambda$ for which the embeddings all have the same critical point $\kappa$. Thus, this constant value $\kappa$ is the critical point of elementary embeddings $j:\<V_\lambda,\in,A\intersect V_\lambda>\to\<V_\theta,\in,A\intersect V_\theta>$ for unboundedly many ordinals $\lambda$. By restricting these embeddings, it follows that $\kappa$ is the critical point of such embedding for every $\lambda>\kappa$, and so $\kappa$ is $A$-extendible by lemma~\ref{Lemma.A-extendible-iff}.

($3\to 4$) Suppose that for every class $A$, there is an $A$-extendible cardinal. Fix any class $A$ and 
and class club $C$. It suffices to find an $A$-extendible cardinal $\kappa$ in $C$. Let $A\oplus C$ be a class coding both $A$ and $C$, and let $\kappa$ be $A\oplus C$-extendible. Let $\lambda>\kappa$ be at least as large as the next element of $C$ above $\kappa$. Since $\kappa$ is $A\oplus C$-extendible, there is an elementary embedding $j:\<V_\lambda,\in,A\intersect V_\lambda,C\intersect\lambda>\to\<V_\theta,\in,A\intersect V_\theta,C\intersect\theta>$, with critical point $\kappa$ and $j(\kappa)>\lambda$. Since $C\intersect j(\kappa)$ has elements above $\kappa$, it follows that $C$ is not bounded below $\kappa$, and so $\kappa\in C$, as desired.

($4\to 1$) Assume statement $4$, and suppose that $\mathcal{M}$ is a proper class of structures in a first-order language $\mathcal{L}$. Let $\kappa$ be an $\mathcal{M}$-extendible cardinal above the size of the language $\mathcal{L}$. Let $\lambda$ be any ordinal above $\kappa$ for which there is an element $M\in\mathcal{M}$ of rank $\lambda$. Since $\kappa$ is $\mathcal{M}$-extendible, there is an ordinal $\theta$ and an elementary embedding $j:\<V_{\lambda+1},\in,\mathcal{M}\intersect V_{\lambda+1}>\to \<V_{\theta+1},\in,\mathcal{M}\intersect V_{\theta+1}>$ with critical point $\kappa$ and $\lambda<j(\kappa)$. Note that $j(\lambda)=\theta$ and so $\lambda<\theta$. Since $j(M)\in\mathcal{M}$ is a structure of rank $\theta$ in $\mathcal{M}$, it is therefore not identical with $M$. Since the language is fixed pointwise by $j$, it follows that $j\restrict M:M\to j(M)$ is an elementary embedding between distinct elements of $\mathcal{M}$, thus verifying this instance of the \Vopenka\ principle.

($3\iff 5$) The forward implication is immediate by the remarks after the proof of lemma~\ref{Lemma.A-extendible-iff}. For the converse, fix any class $A$, and apply statement $5$ to the class $A\oplus V=(\{0\}\times A)\union(\{1\}\times V)$. Thus, for all sufficiently large ordinals $\lambda$ there is a transitive class $M$ and nontrivial elementary embedding $j:V\to M$ with critical point $\kappa<\lambda$ and $\lambda<j(\kappa)$, and the final condition of statement $5$ for the class $A\oplus V$ amounts to $j(A\intersect V_\lambda)=A\intersect V_{j(\lambda)}$ and $j(V_\lambda)=V_{j(\lambda)}$. This latter condition implies $V_{j(\lambda)}\of M$ and so $\kappa$ is $A$-extendible simply by restricting the embedding to $V_\lambda$.
\end{proof}

Essentially identical arguments work in \ZFC\ with the first-order \Vopenka\ scheme, by considering only definable classes:  

\begin{theorem}\label{Theorem.VSequivalents}
The following are equivalent over $\ZFC$ set theory.
 \begin{enumerate}
  \item The \Vopenka\ scheme.
  \item For every definable class $A$ and all sufficiently large $\lambda$, there is an ordinal $\theta>\lambda$ and an elementary embedding $j:\<V_\lambda,\in,A\intersect V_\lambda>\to\<V_\theta,\in,A\intersect V_\theta>$.
  \item For every definable class $A$, there is an $A$-extendible cardinal.
  \item For every definable class $A$, there is a definably stationary proper class of $A$-extendible cardinals.
  \item For every definable class $A$ and all sufficiently large ordinals $\lambda$, there is a definable transitive class $M$ and a definable nontrivial elementary embedding $j:V\to M$ with critical point below $\lambda$, such that $j(A\intersect V_\lambda)=A\intersect V_{j(\lambda)}$.
 \end{enumerate}
\end{theorem}

One may also extract a version of the theorem for \Vopenka\ cardinals. Namely, a cardinal $\delta$ is a {\df Vopenka cardinal}, if $\delta$ is inaccessible and for every set $\mathcal{M}\of V_\delta$ of $\delta$ many first-order $\mathcal{L}$-structures, with $\mathcal{L}$ of size less than $\delta$, there are structures $M\neq N$ in $\mathcal{M}$ with an elementary embedding $j:M\to N$ between them. This is equivalent to saying that the structure $\<V_\delta,\in,V_{\delta+1}>$, a model of Kelley-Morse set theory whose sets are the elements of $V_\delta$ and whose classes include all subsets of $V_\delta$, is a model of the \Vopenka\ principle. For example:

\begin{theorem}\label{Theorem.VP-cardinal-equivalents}
 The following are equivalent, for any inaccessible cardinal $\delta$. 
 \begin{enumerate}
  \item $\delta$ is a \Vopenka\ cardinal.
  \item For every $A\of V_\delta$ and all sufficiently large $\lambda<\delta$, there is an ordinal $\theta$ with $\lambda<\theta<\delta$ and an elementary embedding $j:\<V_\lambda,\in,A\intersect V_\lambda>\to\<V_\theta,\in,A\intersect V_\theta>$.
  \item For every $A\of V_\delta$, there is a $(\ltdelta,A)$-extendible cardinal. That is, there is $\kappa<\delta$ such that for every $\lambda$ with $\kappa<\lambda<\delta$, there is an elementary embedding $j:\<V_\lambda,\in,A\intersect V_\lambda>\to\<V_\theta,\in,A\intersect V_\theta>$, with critical point $\kappa$ and $\lambda<j(\kappa)$. 
  \item For every $A\of V_\delta$, there is a stationary set of such $(\ltdelta,A)$-extendible cardinals below $\delta$.
 \end{enumerate}
\end{theorem}

\noindent The proof simply follows the same argument as in theorem~\ref{Theorem.VPequivalents}.

As an analogue of the \Vopenka\ scheme at this level, let us define that $\delta$ is a {\df \Vopenka-scheme cardinal}, if $V_\delta$ satisfies \ZFC\ plus the \Vopenka\ scheme. The difference is whether the classes are definable or not and also whether inaccessibility is required. We similarly get a list of equivalents in analogy with theorem~\ref{Theorem.VSequivalents}.

\begin{theorem}\label{Theorem.Vopenka-scheme-cardinal-equivalents}
 The following are equivalent for any cardinal $\delta$. 
 \begin{enumerate}
  \item $\delta$ is a \Vopenka-scheme cardinal.
  \item For every $A\of V_\delta$ definable in $\<V_\delta,\in>$ and all sufficiently large $\lambda<\delta$, there is an ordinal $\theta<\delta$ and an elementary embedding $j:\<V_\lambda,\in,A\intersect V_\lambda>\to\<V_\theta,\in,A\intersect V_\theta>$.
  \item For every $A\of V_\delta$ definable in $\<V_\delta,\in>$, there is a $(\ltdelta,A)$-extendible cardinal. That is, there is $\kappa<\delta$ such that for every $\lambda$ with $\kappa<\lambda<\delta$, there is an elementary embedding $j:\<V_\lambda,\in,A\intersect V_\lambda>\to\<V_\theta,\in,A\intersect V_\theta>$, with critical point $\kappa$ and $\lambda<j(\kappa)$.
  \item For every such definable $A\of V_\delta$, there is a definably stationary set of such $(\ltdelta,A)$-extendible cardinals below $\delta$.
 \end{enumerate}
\end{theorem}

Let us briefly call attention to a consequence of statement 4 in each of the theorems we have proved in this section. Namely, the \Vopenka\ principle implies that $\Ord$ is Mahlo, meaning that every class club $C\of\Ord$ contain a regular cardinal, since indeed, every such club contains an $A$-extendible cardinal, for any class $A$, and such cardinals are supercompact and more. Similarly, the \Vopenka\ scheme implies that $\Ord$ is definably Mahlo, the scheme asserting that every definable class club $C\of\Ord$ contains a regular cardinal, and indeed, it contains an $A$-extendible cardinal for any definable class $A$. The same idea shows that every \Vopenka\ cardinal is Mahlo, and every every \Vopenka-scheme cardinal is definably Mahlo.

\begin{corollary}\label{Corollary.Every-Vopenka-cardinal-is-limit-of-Vopenka-scheme-cardinals}
 Every \Vopenka\ cardinal $\delta$ has a club set of \Vopenka\ scheme cardinals below $\delta$. In particular, there is a stationary set of inaccessible \Vopenka\ scheme cardinals below $\delta$, and indeed, a stationary set of $(\ltdelta,A)$-extendible \Vopenka\ scheme cardinals below $\delta$, for any particular $A\of V_\kappa$.
\end{corollary}

\begin{proof}
Suppose that $\delta$ is a \Vopenka\ cardinal. The collection of $\gamma<\delta$ for which $V_\gamma\elesub V_\delta$ is club in $\delta$. Since any instance of the \Vopenka\ scheme is first-order expressible, it follows that all such $\gamma$ are \Vopenka\ scheme cardinals. And since by theorem~\ref{Theorem.VP-cardinal-equivalents} statement $4$, the collection of $(\ltdelta,A)$-extendible cardinals below $\delta$ is stationary, by intersecting with the club it follows that there is a stationary set of $(\ltdelta,A)$-extendible \Vopenka\ scheme cardinals below $\delta$.
\end{proof}

In particular, the existence of a \Vopenka\ cardinal is strictly stronger in consistency strength over \ZFC\ than the existence of a \Vopenka\ scheme cardinal, and indeed, stronger than an extendible \Vopenka\ scheme cardinal. Since the \Vopenka\ cardinals are analogous to the \Vopenka\ principle in the way that \Vopenka\ scheme cardinals are analogous to the \Vopenka\ scheme, this might suggest that the \Vopenka\ principle should be stronger in consistency strength than the \Vopenka\ scheme. But that conclusion would be incorrect, as the main results of this article show that \GBC\ plus the \Vopenka\ principle is in fact equiconsistent with and indeed conservative over \ZFC\ plus the \Vopenka\ scheme. For this reason, it is more correct to say that \Vopenka\ cardinals are analogous to Kelley-Morse \KM\ set theory plus the \Vopenka\ principle than to \GBC\ plus the \Vopenka\ principle, and the former theory is strictly stronger than the latter, for essentially similar reasons as in the proof of corollary~\ref{Corollary.Every-Vopenka-cardinal-is-limit-of-Vopenka-scheme-cardinals}.

\section{Separating the principle from the scheme}

Let us now establish the first part of the main theorem by proving theorem~\ref{Theorem.VS+notVP}, which shows that the two \Vopenka\ axioms, if consistent, are not equivalent. This result was part of my answer~\cite{Hamkins2010.MO46538:Can-Vopenkas-principle-be-violated-definably?} to a question posted on MathOverflow by Mike Shulman, who had inquired whether there would always be a definable counterexample to the \Vopenka\ principle, whenever it should happen to fail. I interpret the question as asking whether the \Vopenka\ scheme is necessarily equivalent to the \Vopenka\ principle, and the answer is negative.

\begin{theorem}\label{Theorem.VS+notVP}
 If the \Vopenka\ scheme holds, then there is a class forcing extension $V[G]$, not adding sets, in which the \Vopenka\ scheme continues to hold, but the \Vopenka\ principle fails. 
\end{theorem}

\begin{proof}
Work in \GBC\ set theory, and assume that the \Vopenka\ scheme holds. Force by initial segments to add a club $C\of\Ord$ avoiding the regular cardinals, and let $V[C]$ be the \GBC\ model arising as the forcing extension. For every cardinal $\lambda$, the collection of conditions reaching above $\lambda$ is $\leqlambda$-closed, and so this forcing adds no new sets. Thus, the sets of the forcing extension $V[C]$ are exactly the sets in $V$, but the classes are those definable in the structure $\<V,\in,C>$. Since the first-order definable classes (using only set parameters) are exactly the same in $V[C]$ as in $V$, we have an $A$-extendible cardinal for any such definable class $A$ in $V[C]$. Consequently, the first-order \Vopenka\ scheme continues to hold in $V[C]$. Meanwhile, since the class $C$ destroys ``$\Ord$  is Mahlo'' with respect to the new non-definable classes, such as $C$ itself, it follows by the observations at the end of the previous section that the \Vopenka\ principle fails in $V[C]$, as desired.
\end{proof}

A similar argument applies to the definability-stratification of the \Vopenka\ scheme into definable levels. Specifically, by adding a class club avoiding the $C^{(n)}$-cardinals, one can force the failure of the $\Sigma_{n+1}$-\Vopenka\ scheme, while preserving the $\Sigma_n$-\Vopenka\ scheme. 

\section{Conservativity of the principle over the scheme}

In this final section, I'd like to prove the main conservativity result, namely, that the \Vopenka\ principle is conservative over the \Vopenka\ scheme for first-order statements in the language of set theory. It follows that the two \Vopenka\ axioms are equiconsistent over \GBC.

\begin{theorem}\label{Theorem.Forcing-GBC+VP}
 If \ZFC\ and the \Vopenka\ scheme holds, then there is a class forcing extension, adding classes but no sets, in which \GBC\ and the \Vopenka\ principle holds.
\end{theorem}

\begin{proof}
Assume that \ZFC\ and the \Vopenka\ scheme holds in $\<V,\in>$. Let us force the global axiom of choice by adding a $V$-generic Cohen class of ordinals $G\of\Ord$, using the class forcing $\Add(\Ord,1)$, whose conditions are elements of $2^{\lt\Ord}$, each describing a possible initial segment of $G$. The class $G$ is $V$-generic in the sense of meeting every definable dense subclass of this forcing. Since the forcing is $\kappa$-closed for every cardinal $\kappa$, the forcing extension $V[G]$ has the same sets as $V$. The classes of $V[G]$ are those definable in the structure $\<V,\in,G>$, allowing the predicate $G$. It is well-known that this is a model of \GBC, and indeed this is how one proves that \GBC\ is conservative over \ZFC.

Let me start by proving that there is a $G$-extendible cardinal in $V[G]$. First, note that in $V$ there is a stationary proper class of extendible cardinals. By density, there must be an extendible cardinal $\kappa$ for which $G\intersect\kappa$ is stretchable, since by theorem~\ref{Theorem.Stretchable-set} and the remarks after it, any given condition can be extended to a stretchable set at any desired extendible cardinal above that condition. I claim that this $\kappa$ is $G$-extendible in $V[G]$, and in fact, the condition $g=G\intersect\kappa$ will force that $\kappa$ is $G$-extendible. To see this, consider any $\lambda>\kappa$ and any stronger condition $p$ extending $g$. By making $\lambda$ larger, I may assume without loss that $p\of\lambda$. By stretchability, there is an elementary embedding $j:V_{\lambda+1}\to V_{\theta+1}$ with critical point $\kappa$ and $j(\kappa)>\lambda$ for which $j(g)\intersect \lambda=p$. Let $q=j(g)$ and let $r=j(q\intersect\lambda)$. Thus, $j\restrict V_\lambda:\<V_\lambda,\in,q\intersect\lambda>\to\<V_\theta,\in,r>$ is an elementary embedding with critical point $\kappa$ and $\lambda<j(\kappa)$. But note also---and this is the key point---that since $q\intersect\lambda$ extends $g$, we have that $j(q\intersect\lambda)$ extends $j(g)$, which is $q$. So $r$ extends $q$, and in particular, $r\intersect\lambda=q\intersect\lambda$. So we may write the elementary embedding as $j\restrict V_\lambda:\<V_\lambda,\in,r\intersect\lambda>\to\<V_\theta,\in,r\intersect\theta>$. Thus, the condition $r$ forces that $\kappa$ is $G$-extendible at least to degree $\lambda$. But $r$ extended $p$, and so we have proved that for ordinal $\lambda>\kappa$, it is dense that $\kappa$ is $G$-extendible to degree $\lambda$. So by genericity, $\kappa$ is fully $G$-extendible.
 
We may now easily augment the previous argument with any class $A$ that is first-order definable in $\<V,\in>$, in order to find an $A\oplus G$-extendible cardinal $\kappa$ in $V[G]$. Namely, we first find an $A$-extendible cardinal $\kappa$ for which $g=G\intersect\kappa$ is $A$-stretchable. This exists by genericity, using theorem~\ref{Theorem.Stretchable-set} and the fact that there are unboundedly many $A$-extendible cardinals. Now argue that $g$ forces that $\kappa$ is $A\oplus G$-extendible in $V[G]$. For any condition $p$ extending $g$ and ordinal $\lambda>\kappa$, we may assume $p\of\lambda$ and then by $A$-stretchability find an elementary embedding $j:\<V_{\lambda+1},\in,A\intersect V_{\lambda+1}>\to\<V_{\theta+1},\in,A\intersect V_{\theta+1}>$ with critical point $\kappa$ and $\lambda<j(\kappa)$, for which $j(g)\intersect\lambda=p$. Let $q=j(g)$ and $r=j(q\intersect\lambda)$. So again we have $r\intersect\lambda=q\intersect\lambda$ and $j\restrict V_\lambda:\<V_\lambda,\in,A\intersect V_\lambda,r\intersect\lambda>\to \<V_\theta,\in,A\intersect V_\theta,r\intersect\theta>$ is an elementary embedding with critical point $\kappa$ and $\lambda<j(\kappa)$. So $r$ forces that $\kappa$ is $A\oplus G$-extendible to degree $\lambda$. So there are a dense class of such conditions and therefore by genericity, the cardinal $\kappa$ is $A\oplus G$-extendible to every degree.

The argument until now considered only classes that were definable in the ground model and $G$ itself. So now let us finally consider the case where $A$ is a class that is definable from $G$ in $\<V,\in,G>$. Suppose that $A=\set{a\mid V[G]\satisfies\varphi(a,z,G)}$ for some first-order formula $\varphi$, in which a predicate for $G$ may appear, and parameter $z$. By the definability of the forcing relation, we know that the classes $\dot A=\set{\<a,p>\mid p\forces\varphi(a,z,\dot G)}$ and
$\bar A=\set{\<a,p>\mid p\forces\neg\varphi(a,z,\dot G)}$ are definable in the ground model $V$. By the previous paragraph, there is a cardinal $\kappa$ that is $(\dot A\oplus\bar A\oplus G)$-extendible. Fix any $\lambda$, and let $\bar\lambda>\lambda$ be large enough so that for every $a\in V_\lambda$, there is a condition $p=G\intersect\alpha$ with $\alpha<\bar\lambda$ that decides $\varphi(a,z,\dot G)$. Let $$j:\<V_{\bar\lambda},\in,\dot A\intersect V_{\bar\lambda},\bar A\intersect V_{\bar\lambda},G\intersect\bar\lambda>\to\<V_{\bar\theta},\in,\dot A\intersect V_{\bar\theta},\bar A\intersect V_{\bar\theta},G\intersect\bar\theta>$$ witness the $(\dot A\oplus\bar A\oplus G)$-extendibility of $\kappa$. The assumption on $\bar\lambda$ ensures that $A\intersect V_\lambda$ and $A\intersect V_\theta$, where $\theta=j(\lambda)$, are definable respectively in these two structures. Thus, $j\restrict V_\lambda:\<V_\lambda,\in,A\intersect V_\lambda>\to \<V_\theta,\in,A\intersect V_\theta>$ is an elementary embedding witnessing this instance of the $A$-extendibility of $\kappa$. So $\kappa$ is $A$-extendible, and we have therefore proved that every class $A$ in $V[G]$ admits an $A$-extendible cardinal. By theorem~\ref{Theorem.VPequivalents}, therefore, $V[G]$ is a model of the \Vopenka\ principle, as desired.
\end{proof}

It seems likely to me that the stretchability idea will be central in future $A$-extendibility lifting arguments, as it provides an extendibility analogue for the master condition technique.

\begin{corollary}\label{Corollary.VP-conservative-over-VS}
 The two \Vopenka\ axioms have exactly the same first-order consequences in the language of set theory. Specifically, \GBC\ plus the \Vopenka\ principle proves a first-order statement $\phi$ in the language of set theory if and only if \ZFC\ plus the \Vopenka\ scheme proves $\phi$. 
 $$\GBC+\VP\proves\phi\qquad\text{if and only if}\qquad\ZFC+\VS\proves\phi$$
\end{corollary}

\noindent 
In other words, \GBC\ plus the \Vopenka\ principle is conservative over \ZFC\ plus the \Vopenka\ scheme. 

\begin{proof}
Since the \Vopenka\ principle implies the \Vopenka\ scheme, the converse direction is immediate. For the forward implication, theorem~\ref{Theorem.Forcing-GBC+VP} shows that every model of \ZFC\ plus the \Vopenka\ scheme can be expanded to a model of \GBC\ plus the \Vopenka\ principle, with the same first-order part. So any first-order statement that fails in some model of \ZFC\ plus the \Vopenka\ scheme also fails in some model of \GBC\ plus the \Vopenka\ principle. So if a first-order statement is provable in \GBC\ plus the \Vopenka\ principle, then it must hold in all models of \ZFC\ plus the \Vopenka\ scheme, and hence it is provable in that theory. 
\end{proof}

In particular, if one of the theories proves a contradiction, then so does the other, and therefore:

\begin{corollary}\label{Corollary.VP-equiconsistent-VS}
 \GBC\ plus the \Vopenka\ principle is equiconsistent with \ZFC\ plus the \Vopenka\ scheme. 
\end{corollary}

\bibliographystyle{alpha}
\bibliography{MathBiblio,HamkinsBiblio,WebPosts}

\end{document}